\documentclass[11pt]{article}

\usepackage{amsmath}
\usepackage{amsthm}
\usepackage{amsfonts}
\usepackage{amssymb}
\usepackage{enumerate}
\usepackage{tikz}
\usepackage[margin=1.1in]{geometry}
\usepackage[hidelinks]{hyperref}

\newtheorem{thm}{Theorem}[section]
\newtheorem{cor}[thm]{Corollary}
\newtheorem{lem}[thm]{Lemma}
\newtheorem{prop}[thm]{Proposition}
\theoremstyle{definition}
\newtheorem{defn}[thm]{Definition}
\theoremstyle{remark}

\DeclareMathOperator{\card}{card}
\DeclareMathOperator{\Int}{Int}
\newcommand{\R}{\mathbb{R}}
\newcommand{\Z}{\mathbb{Z}}
\newcommand{\Q}{Q}

\begin{document}

\title{Multiple Lattice Packings and Coverings of the Plane with Triangles}
\author{Kirati Sriamorn\\[4pt]
\small Department of Mathematics and Computer Science, Faculty of Science,\\
\small Chulalongkorn University, Bangkok, Thailand\\
\small \texttt{kirati.s@chula.ac.th}}
\date{}
\maketitle

\begin{abstract}
Given a convex disk $K$ and a positive integer $j$, let $\delta_L^j(K)$ and $\vartheta_L^j(K)$ denote the $j$-fold lattice packing density and the $j$-fold lattice covering density of $K$, respectively. We prove that
\[
\delta_L^j(T)=\frac{2j^2}{2j+1}\qquad\text{and}\qquad \vartheta_L^j(T)=\frac{2j+1}{2}
\]
for every triangle $T$ and every positive integer $j$. We also determine all lattices attaining these densities; in both cases their number is $(2j+1)\prod_{p\mid 2j+1}\left(1-\frac{2}{p}\right)$, where the product is over the distinct primes dividing $2j+1$.
\end{abstract}

\section{Introduction}

Throughout, $|S|$ denotes the area (Lebesgue measure) of a measurable set $S\subset\R^2$, $\Int(S)$ its interior and $\card\{S\}$ its cardinality. A \emph{lattice} is a set $\Lambda=\{k v_1+l v_2: k,l\in\Z\}$, where $v_1,v_2$ are linearly independent; $d(\Lambda)=|\det(v_1,v_2)|$ is its \emph{determinant}. For a bounded measurable set $K$ we write $K+\Lambda$ for the family $\{K+v:v\in\Lambda\}$.

Let $j$ be a positive integer. The family $K+\Lambda$ is a \emph{$j$-fold lattice packing} if every point of $\R^2$ lies in the interiors of at most $j$ translates $K+v$, $v\in\Lambda$; it is a \emph{$j$-fold lattice covering} if every point lies in at least $j$ translates; and it is an \emph{exact $j$-fold lattice tiling} if every point lies in exactly $j$ translates. In the last case we call $K$ an \emph{exact $j$-fold tile}. For $j=1$ we simply speak of packings, coverings and tilings. The density of $K+\Lambda$ is $|K|/d(\Lambda)$, and
\[
\delta_L^j(K)=\sup_{\Lambda}\frac{|K|}{d(\Lambda)},\qquad
\vartheta_L^j(K)=\inf_{\Lambda}\frac{|K|}{d(\Lambda)},
\]
where the supremum (infimum) is taken over all lattices $\Lambda$ for which $K+\Lambda$ is a $j$-fold lattice packing (covering). Replacing lattice arrangements by arbitrary families of translates, or of congruent copies, of $K$ one defines in the usual way (via upper and lower densities) the quantities $\delta_T^j(K)$, $\vartheta_T^j(K)$ and $\delta^j(K)$, $\vartheta^j(K)$; see \cite{brass_moser_pach,chuanming}. Writing $\delta_L(K)=\delta_L^1(K)$ and $\vartheta_L(K)=\vartheta_L^1(K)$, one clearly has
\[
j\delta_L(K)\le\delta_L^j(K)\le\delta_T^j(K)\le\delta^j(K)\le j
\le\vartheta^j(K)\le\vartheta_T^j(K)\le\vartheta_L^j(K)\le j\vartheta_L(K).
\]
Moreover, if $\phi(x)=Ax+b$ is a non-singular affine map, then $K+\Lambda$ is a $j$-fold lattice packing (covering, exact tiling) if and only if $\phi(K)+A\Lambda$ is, and both arrangements have the same density. Hence $\delta_L^j$ and $\vartheta_L^j$ are affine invariants.

\medskip
\noindent\textbf{Background.} For $j=1$ and $K$ a triangle, Fáry \cite{fary} proved that $\delta_L(T)=\frac23$ and $\vartheta_L(T)=\frac32$. Multiple lattice arrangements have been studied mostly for centrally symmetric disks. Blundon \cite{blundon1,blundon2} determined $\vartheta_L^j(B^2)$ and $\delta_L^j(B^2)$ for small values of $j$, where $B^2$ is the unit disk. In 1972 Dumir and Hans-Gill \cite{dumir1,dumir2} proved that $\delta_L^2(C)=2\delta_L(C)$ and $\vartheta_L^2(C)=2\vartheta_L(C)$ for every centrally symmetric convex disk $C$. Following an idea of J.~Pach of decomposing multiple arrangements into simpler ones (see \cite{brass_moser_pach}), G.~Fejes T\'oth \cite{fejes} showed that every $3$-fold lattice packing can be decomposed into three lattice packings and every $4$-fold lattice packing into two $2$-fold lattice packings; consequently $\delta_L^3(C)=3\delta_L(C)$ and $\delta_L^4(C)=4\delta_L(C)$. For multiple coverings of the plane with translates of a triangle in a different (decomposability) sense see Tardos and G.~T\'oth \cite{gabor}. A survey of the known results, in particular on $\delta_L^j(B^2)$ and $\vartheta_L^j(B^2)$, is given in \cite{chuanming}.

\medskip
\noindent\textbf{Main results.} In this paper we determine $\delta_L^j(T)$ and $\vartheta_L^j(T)$ for every triangle $T$ and every $j$. For an integer $m$ let
\[
\Lambda(m,j)=\{k(1,m)+l(0,2j+1): k,l\in\Z\}=\{(a,b)\in\Z^2:\ b\equiv ma \pmod{2j+1}\}.
\]
(The two descriptions agree because $(a,b)=a(1,m)+\frac{b-ma}{2j+1}(0,2j+1)$.) Clearly $\Lambda(m,j)$ depends only on the residue of $m$ modulo $2j+1$, and different residues give different lattices, since $(1,m')\in\Lambda(m,j)$ if and only if $m'\equiv m\pmod{2j+1}$.

\begin{thm}\label{main1}
For every triangle $T$ and every positive integer $j$,
\[
\delta_L^j(T)=\frac{2j^2}{2j+1}\qquad\text{and}\qquad\vartheta_L^j(T)=\frac{2j+1}{2}.
\]
\end{thm}

Let $\Delta_L^j(K)$ ($\Theta_L^j(K)$) denote the set of lattices $\Lambda$ such that $K+\Lambda$ is a $j$-fold lattice packing (covering) of density $\delta_L^j(K)$ (respectively $\vartheta_L^j(K)$).

\begin{thm}\label{main_lattices}
Let $T_0$ be the triangle with vertices $(0,0)$, $(1,0)$, $(0,1)$, and let $M_j$ be the set of integers $m\in\{1,\ldots,2j+1\}$ with $\gcd(m,2j+1)=\gcd(m+1,2j+1)=1$. Then
\[
\Delta_L^j(T_0)=\Bigl\{\tfrac{1}{2j}\Lambda(m,j): m\in M_j\Bigr\}
\qquad\text{and}\qquad
\Theta_L^j(T_0)=\Bigl\{\tfrac{1}{2j+1}\Lambda(m,j): m\in M_j\Bigr\}.
\]
In other words, $\Lambda\in\Delta_L^j(T_0)$ if and only if $\Lambda$ is generated by $\bigl(\frac{1}{2j},\frac{m}{2j}\bigr)$ and $\bigl(0,\frac{2j+1}{2j}\bigr)$ for some $m\in M_j$, and $\Lambda\in\Theta_L^j(T_0)$ if and only if $\Lambda$ is generated by $\bigl(\frac{1}{2j+1},\frac{m}{2j+1}\bigr)$ and $(0,1)$ for some $m\in M_j$.
\end{thm}

By affine invariance, Theorem~\ref{main_lattices} immediately describes the optimal lattices for an arbitrary triangle $T=\phi(T_0)$, $\phi(x)=Ax+b$: they are the images $A\Lambda$ of the lattices above. In particular, counting the elements of $M_j$ (see Section~\ref{sec:proofs}) we obtain:

\begin{cor}\label{cor_number}
For every triangle $T$ and every positive integer $j$,
\[
\card\{\Delta_L^j(T)\}=\card\{\Theta_L^j(T)\}=(2j+1)\prod_{p\mid 2j+1}\Bigl(1-\frac{2}{p}\Bigr),
\]
where the product is over the distinct primes $p$ dividing $2j+1$.
\end{cor}

For instance, for $j=1$ the optimal lattices are unique (Fáry's lattices), for $j=2$ there are three of them, and for $j=4$ again three.

\medskip
\noindent\textbf{Translative arrangements.} The results of this paper were first announced in the preprint \cite{arxiv}, and they have since been used to determine the corresponding translative densities. For $j=1$ it is classical that $\delta_T(K)=\delta_L(K)$ for every convex disk $K$ (see \cite{rogers,brass_moser_pach}), whereas the equality $\vartheta_T(K)=\vartheta_L(K)$ is known only for special classes of disks. For triangles it was proved by Januszewski \cite{januszewski}, and it was extended to quarter-convex disks (a class containing all triangles and all convex quadrilaterals) in \cite{sriamorn_xue}. For multiple arrangements of a triangle $T$, building on Theorem~\ref{main1}, it was shown in \cite{sriamorn_wetayawanich} that $\vartheta_T^j(T)=\vartheta_L^j(T)=\frac{2j+1}{2}$, and in \cite{sriamorn_packing} that $\delta_T^j(T)=\delta_L^j(T)=\frac{2j^2}{2j+1}$. Thus, for triangles, the $j$-fold translative packing and covering densities are also determined by Theorem~\ref{main1}.

\medskip
\noindent\textbf{Outline of the proof.} The key objects are \emph{stair polygons}, i.e.\ unions of half-open rectangles $[x_i,x_{i+1})\times[0,y_i)$ with $0=x_0<x_1<\cdots<x_{r+1}$ and $y_0>y_1>\cdots>y_r>0$ (Definition~\ref{def_stair}), and in particular the stair polygon
\[
S(j)=\bigcup_{i=0}^{2j-1}[i,i+1)\times[0,2j-i)
\]
(Figure~\ref{fig_Sj}). The proof consists of three steps.

\begin{enumerate}[(1)]
\item \emph{Reduction to stair polygons} (Sections~\ref{subsec_SjLambda} and \ref{subsec_steps}). To every lattice $\Lambda$ we associate a stair polygon $S_j(\Lambda)$, defined by counting lattice points in quadrants, such that $S_j(\Lambda)+\Lambda$ is an exact $j$-fold lattice tiling, $S_j(\Lambda)$ has at most $2j-1$ steps, $S_j(\Lambda)\subset\lambda T_0$ whenever $\lambda T_0+\Lambda$ is a $j$-fold covering, and $\Int(\lambda T_0)\subset S_j(\Lambda)$ whenever $\lambda T_0+\Lambda$ is a $j$-fold packing. Since $|S_j(\Lambda)|=j\,d(\Lambda)$, the density of a $j$-fold lattice covering (packing) with $T_0$ is at least (at most) $j|T_0|$ divided by the maximal (minimal) area of a stair polygon with at most $2j-1$ steps inscribed in (circumscribed about) $T_0$.
\item \emph{An extremal problem} (Section~\ref{subsec_extremal}). By an elementary inequality, these extremal areas are $\frac{j}{2j+1}$ and $\frac{2j+1}{4j}$, and they are attained only by $\frac{1}{2j+1}S(j)$ and $\frac{1}{2j}S(j)$, respectively. Since $S(j)$ is an exact $j$-fold tile, Theorem~\ref{main1} follows.
\item \emph{Tilings with $S(j)$} (Section~\ref{sec_tiling}). By step (2), the optimal lattices are exactly the lattices $\Lambda$ for which $\frac{1}{2j}S(j)+\Lambda$ (respectively $\frac{1}{2j+1}S(j)+\Lambda$) is an exact $j$-fold tiling. These are determined by the following result, which may be of independent interest.
\end{enumerate}

\begin{thm}\label{thm_Sj}
Let $j$ be a positive integer and $\Lambda$ a lattice. Then $S(j)+\Lambda$ is an exact $j$-fold lattice tiling of $\R^2$ if and only if $\Lambda=\Lambda(m,j)$ for some integer $m$ with $\gcd(m,2j+1)=\gcd(m+1,2j+1)=1$.
\end{thm}

Section~\ref{sec_prelim} collects preliminaries, Section~\ref{sec_tiling} proves Theorem~\ref{thm_Sj} (it does not depend on the rest of the paper), and Section~\ref{sec:proofs} carries out steps (1) and (2) and completes the proofs of Theorems~\ref{main1}, \ref{main_lattices} and Corollary~\ref{cor_number}.

\begin{figure}[ht]
\centering
\begin{tikzpicture}[scale=0.62]
  \draw[step=1,gray!25,very thin] (-0.5,-0.5) grid (6.5,6.5);
  \draw[->] (-0.5,0)--(6.7,0) node[right]{$x$};
  \draw[->] (0,-0.5)--(0,6.7) node[above]{$y$};
  \fill[blue!15] (0,0)--(0,4)--(1,4)--(1,3)--(2,3)--(2,2)--(3,2)--(3,1)--(4,1)--(4,0)--cycle;
  \draw[thick,blue] (0,0)--(0,4)--(1,4)--(1,3)--(2,3)--(2,2)--(3,2)--(3,1)--(4,1)--(4,0)--cycle;
  \draw[dashed] (0,5)--(5,0);
  \foreach \a in {0,...,6}{\foreach \b in {0,...,6}{
     \pgfmathtruncatemacro{\r}{mod(\b-2*\a+50,5)}
     \ifnum\r=0 \fill (\a,\b) circle (2.6pt); \fi}}
  \foreach \t in {1,2,3,4,5}{\node[below] at (\t,0) {\scriptsize $\t$};\node[left] at (0,\t) {\scriptsize $\t$};}
  \node at (1.1,1.5) {$S(2)$};
\end{tikzpicture}
\caption{The stair polygon $S(2)$ together with the points of $\Lambda(2,2)$ (black dots). The dashed line is $x+y=5$; the outer corners of $S(j)$ lie on $x+y=2j+1$ and its inner corners on $x+y=2j$.}\label{fig_Sj}
\end{figure}

\section{Preliminaries}\label{sec_prelim}

\subsection{Lattice arrangements}

\begin{lem}\label{lem_coset}
Let $K\subset\R^2$ be bounded and measurable and let $\Lambda$ be a lattice. For $p\in\R^2$ and $v\in\Lambda$ we have $p\in K+v$ if and only if $p-v\in K$; hence the translates of $K$ containing $p$ correspond bijectively to the points of $(p+\Lambda)\cap K$. Consequently:
\begin{enumerate}[(a)]
\item $K+\Lambda$ is a $j$-fold lattice packing if and only if $\card\{(p+\Lambda)\cap\Int(K)\}\le j$ for all $p\in\R^2$;
\item $K+\Lambda$ is a $j$-fold lattice covering if and only if $\card\{(p+\Lambda)\cap K\}\ge j$ for all $p\in\R^2$;
\item $K+\Lambda$ is an exact $j$-fold lattice tiling if and only if $\card\{(p+\Lambda)\cap K\}=j$ for all $p\in\R^2$. In this case $|K|=j\,d(\Lambda)$.
\end{enumerate}
\end{lem}

\begin{proof}
Only the last assertion needs a proof. Let $P$ be a half-open fundamental parallelogram of $\Lambda$. Integrating $\sum_{v\in\Lambda}\mathbf 1_K(p-v)=j$ over $p\in P$ gives $j\,d(\Lambda)=\sum_{v\in\Lambda}|(P-v)\cap K|=|K|$, because the sets $P-v$, $v\in\Lambda$, tile the plane.
\end{proof}

\subsection{Stair polygons}

Let $\Q=[0,\infty)^2$. For points of $\R^2$ we write $p\le q$ if $p_x\le q_x$ and $p_y\le q_y$.

\begin{defn}\label{def_stair}
Let $r\ge 0$ be an integer. A (half-open) \emph{$r$-stair polygon} is a set of the form
\[
S=\bigcup_{i=0}^{r}[x_i,x_{i+1})\times[0,y_i),\qquad 0=x_0<x_1<\cdots<x_{r+1},\quad y_0>y_1>\cdots>y_r>0.
\]
A \emph{stair polygon} is an $r$-stair polygon for some $r\ge0$. We call the points $(x_i,y_i)$, $1\le i\le r$, the \emph{inner corners} of $S$, the points $(x_{i+1},y_i)$, $0\le i\le r$, its \emph{outer corners}, and $(0,y_0)$, $(x_{r+1},0)$ its \emph{end points}. We write
\[
Z(S)=\{(x_i,y_i):1\le i\le r\},\qquad Z^*(S)=\{(0,y_0),(x_{r+1},0)\}.
\]
Thus $\card\{Z(S)\}=r$ is the number of steps of $S$.
\end{defn}

See Figure~\ref{fig_stair}. Note that a dilate $\lambda S$ ($\lambda>0$) of an $r$-stair polygon is again an $r$-stair polygon.

\begin{figure}[ht]
\centering
\begin{tikzpicture}[scale=0.9]
  \draw[->] (-0.3,0)--(6.6,0) node[right]{$x$};
  \draw[->] (0,-0.3)--(0,4.4) node[above]{$y$};
  \fill[blue!12] (0,0)--(0,3.6)--(1.4,3.6)--(1.4,2.5)--(2.6,2.5)--(2.6,2)--(4.3,2)--(4.3,0.8)--(5.8,0.8)--(5.8,0)--cycle;
  \draw[thick,blue] (0,0)--(0,3.6)--(1.4,3.6)--(1.4,2.5)--(2.6,2.5)--(2.6,2)--(4.3,2)--(4.3,0.8)--(5.8,0.8)--(5.8,0);
  \foreach \p/\l in {(1.4,2.5)/{(x_1,y_1)},(2.6,2)/{(x_2,y_2)},(4.3,0.8)/{(x_3,y_3)}}{\fill[red] \p circle (2pt); \node[below left] at \p {\scriptsize $\l$};}
  \foreach \p in {(1.4,3.6),(2.6,2.5),(4.3,2),(5.8,0.8)}{\draw[fill=white] \p circle (2pt);}
  \fill (0,3.6) circle (2pt) node[left]{\scriptsize $(0,y_0)$};
  \fill (5.8,0) circle (2pt) node[below]{\scriptsize $(x_4,0)$};
\end{tikzpicture}
\caption{A $3$-stair polygon: inner corners (red), outer corners (white) and end points (black).}\label{fig_stair}
\end{figure}

\begin{lem}\label{lem_stair_basic}
Let $S$ be a stair polygon as in Definition~\ref{def_stair}.
\begin{enumerate}[(a)]
\item If $p\in S$, $q\in\Q$ and $q\le p$, then $q\in S$.
\item $S\cap(0,\infty)^2\subset\Int(S)$.
\item Let $z\in(0,\infty)^2$. Then $z\in Z(S)$ if and only if $z\notin S$ and $z-(\varepsilon,0)\in S$, $z-(0,\varepsilon)\in S$ for all sufficiently small $\varepsilon>0$.
\item Let $y>0$. Then $(0,y)=(0,y_0)$ if and only if $(0,y)\notin S$ and $(0,y-\varepsilon)\in S$ for all sufficiently small $\varepsilon>0$. Similarly, for $x>0$ we have $(x,0)=(x_{r+1},0)$ if and only if $(x,0)\notin S$ and $(x-\varepsilon,0)\in S$ for all sufficiently small $\varepsilon>0$.
\end{enumerate}
\end{lem}

\begin{proof}
(a) is clear, since the heights $y_i$ decrease. (b) Let $(x,y)\in S$ with $x,y>0$, say $x_i\le x<x_{i+1}$ and $y<y_i$. If $x>x_i$, a small square centred at $(x,y)$ lies in $[x_i,x_{i+1})\times[0,y_i)$; if $x=x_i$ then $i\ge1$ and, as $y_{i-1}>y_i$, a small square centred at $(x,y)$ lies in the union of the $(i-1)$-st and $i$-th rectangles.
(c) Let $z=(x,y)$ satisfy the condition. Since $z-(0,\varepsilon)\in S$, we have $x<x_{r+1}$; let $x_i\le x<x_{i+1}$. Since $z-(0,\varepsilon)\in S$ and $z\notin S$, we get $y=y_i$. If $x>x_i$ then $z-(\varepsilon,0)$ lies in the $i$-th column at height $y_i$, hence not in $S$; so $x=x_i$, and $i\ge1$ because $x>0$. The converse is obvious. (d) is proved in the same way.
\end{proof}

\subsection{An order on the plane}

For $p=(p_x,p_y)$ put $\sigma(p)=p_x+p_y$. We write $p\prec q$ if either $\sigma(p)<\sigma(q)$, or $\sigma(p)=\sigma(q)$ and $p_x<q_x$. This is a strict total order on $\R^2$, and it is invariant under translations: $p\prec q$ implies $p+w\prec q+w$ for every $w\in\R^2$. Note that $(0,0)$ is the smallest point of $\Q$ and that, for $q\in\Q$, the set $\{p\in\Q:p\prec q\}$ is bounded.

\section{Multiple lattice tilings with \texorpdfstring{$S(j)$}{S(j)}}\label{sec_tiling}

In this section we prove Theorem~\ref{thm_Sj}. In Section~\ref{subsec_sufficiency} we show that the lattices $\Lambda(m,j)$ with $\gcd(m,2j+1)=\gcd(m+1,2j+1)=1$, and only these among all lattices $\Lambda(m,j)$, give exact $j$-fold tilings with $S(j)$. In Section~\ref{subsec_necessity} we show that no other lattice does. Throughout, $n=2j+1$.

\subsection{The lattices \texorpdfstring{$\Lambda(m,j)$}{Lambda(m,j)}}\label{subsec_sufficiency}

\begin{prop}\label{prop_sufficiency}
Let $m$ be an integer. Then $S(j)+\Lambda(m,j)$ is an exact $j$-fold lattice tiling of $\R^2$ if and only if $\gcd(m,n)=\gcd(m+1,n)=1$.
\end{prop}

\begin{proof}
Let
\[
D=S(j)\cap\Z^2=\{(a,b)\in\Z^2: a\ge0,\ b\ge0,\ a+b\le 2j-1\},
\]
so that $S(j)=D+[0,1)^2$ and $\card\{D\}=jn$. Let $\psi:\Z^2\to\Z_n$, $\psi(a,b)=b-ma \bmod n$. This is a surjective homomorphism with kernel $\Lambda(m,j)$. For $p\in\R^2$ let $\lfloor p\rfloor\in\Z^2$ be the point with coordinates $\lfloor p_x\rfloor,\lfloor p_y\rfloor$. Since $\Lambda(m,j)\subset\Z^2$, for $v\in\Lambda(m,j)$ we have $p+v\in S(j)$ if and only if $\lfloor p\rfloor+v\in D$. Therefore
\[
\card\{(p+\Lambda(m,j))\cap S(j)\}=\card\{\delta\in D:\ \delta-\lfloor p\rfloor\in\Lambda(m,j)\}=f\bigl(\psi(\lfloor p\rfloor)\bigr),
\]
where $f(\rho)=\card\{\delta\in D:\psi(\delta)=\rho\}$ for $\rho\in\Z_n$. By Lemma~\ref{lem_coset}(c) and the surjectivity of $\psi$, $S(j)+\Lambda(m,j)$ is an exact $j$-fold tiling if and only if $f\equiv j$, and since $\sum_\rho f(\rho)=jn$, this holds if and only if $f$ is constant.

Let $D'=D+(0,1)$. Then $D\setminus D'=\{(a,0):0\le a\le n-2\}$ and $D'\setminus D=\{(a,2j-a):0\le a\le n-2\}$. Since $\psi(\delta+(0,1))=\psi(\delta)+1$, counting the points of $D'$ with $\psi$-value $\rho$ in two ways gives
\begin{equation}\label{eq_shift}
f(\rho-1)=f(\rho)-B(\rho)+C(\rho)\qquad(\rho\in\Z_n),
\end{equation}
where
\[
B(\rho)=\card\{0\le a\le n-2:\ -ma\equiv\rho\},\qquad C(\rho)=\card\{0\le a\le n-2:\ -1-(m+1)a\equiv\rho\}
\]
(congruences modulo $n$). By \eqref{eq_shift}, $f$ is constant if and only if $B=C$.

Let $d=\gcd(m,n)$ and $e=\gcd(m+1,n)$. When $a$ runs through $0,1,\ldots,n-1$, the residue $-ma$ takes every value of the subgroup $d\Z_n$ exactly $d$ times, and the omitted value $a=n-1$ contributes $-m(n-1)\equiv m$. Similarly $-1-(m+1)a$ takes every value of the coset $-1+e\Z_n$ exactly $e$ times, and $a=n-1$ contributes $-1+(m+1)\equiv m$. Hence
\[
B=d\cdot\mathbf 1_{d\Z_n}-\mathbf 1_{\{m\}},\qquad C=e\cdot\mathbf 1_{-1+e\Z_n}-\mathbf 1_{\{m\}},
\]
and $B=C$ if and only if $d\cdot\mathbf 1_{d\Z_n}=e\cdot\mathbf 1_{-1+e\Z_n}$. If $d=e=1$, both sides are identically $1$. Conversely, evaluating at $0$ shows that $0\in-1+e\Z_n$, i.e.\ $e=1$; then the right-hand side is identically $1$, which forces $d\Z_n=\Z_n$, i.e.\ $d=1$.
\end{proof}

\subsection{No other lattices}\label{subsec_necessity}

The following simple observation will be used repeatedly. For a stair polygon $S$, a lattice $\Lambda$ and $p\in\R^2$, let
\[
L(p)=\{v\in\Lambda:\ p\in S+v\}=\{v\in\Lambda:\ p-v\in S\}.
\]
If $S+\Lambda$ is an exact $j$-fold tiling, then $\card\{L(p)\}=j$ for all $p$.

\begin{lem}\label{lem_edge}
Let $S$ be a stair polygon and $\Lambda$ a lattice such that $S+\Lambda$ is an exact $j$-fold tiling. Let $p\notin S$.
\begin{enumerate}[(a)]
\item If $p-(0,\varepsilon)\in S$ for all sufficiently small $\varepsilon>0$, then some $v\in L(p)$ satisfies $v_y=p_y$.
\item If $p-(\varepsilon,0)\in S$ for all sufficiently small $\varepsilon>0$, then some $v\in L(p)$ satisfies $v_x=p_x$.
\end{enumerate}
\end{lem}

\begin{proof}
(a) For $v\in L(p)$ we have $p-v\in S\subset\Q$, so $v_y\le p_y$. Suppose $v_y<p_y$ for all $v\in L(p)$, and choose $\varepsilon>0$ smaller than all the numbers $p_y-v_y$ and such that $p-(0,\varepsilon)\in S$. By Lemma~\ref{lem_stair_basic}(a), $p-(0,\varepsilon)-v\in S$ for every $v\in L(p)$, and also for $v=(0,0)$, which is not in $L(p)$ since $p\notin S$. Thus $p-(0,\varepsilon)$ lies in $j+1$ translates of $S$, a contradiction. (b) is symmetric.
\end{proof}

From now on in this section, $\Lambda$ is a lattice such that $S(j)+\Lambda$ is an exact $j$-fold tiling, and $L(p)$ refers to $S=S(j)$. By Lemma~\ref{lem_coset}(c), $d(\Lambda)=|S(j)|/j=2j+1$. We shall use the elementary fact that the determinant of any two linearly independent vectors of $\Lambda$ is a non-zero integer multiple of $d(\Lambda)$. We also note that the reflection $\tau(x,y)=(y,x)$ maps $S(j)$ onto itself, so $S(j)+\tau(\Lambda)$ is an exact $j$-fold tiling as well; this allows us to argue ``by symmetry''.

\begin{lem}\label{lem_axis}
\begin{enumerate}[(a)]
\item $\Lambda$ contains points $(1,y)$ and $(x,1)$ with $-1<x,y\le 2j-1$.
\item If $s>0$ and $(s,0)\in\Lambda$ or $(0,s)\in\Lambda$, then $s\ge2j+1$.
\end{enumerate}
\end{lem}

\begin{proof}
(a) The point $p=(1,2j-1)$ is not in $S(j)$, while $p-(\varepsilon,0)\in S(j)$ for $0<\varepsilon\le1$. By Lemma~\ref{lem_edge}(b) there is $v\in L(p)$ with $v_x=1$. Since $p-v=(0,2j-1-v_y)\in S(j)$, we get $0\le 2j-1-v_y<2j$, so $v=(1,y)$ with $-1<y\le2j-1$. The point $(x,1)$ is obtained by symmetry (consider $p=(2j-1,1)$ and Lemma~\ref{lem_edge}(a)).

(b) If $(s,0)\in\Lambda$, then together with $(x,1)\in\Lambda$ from (a) we get $s=\det\bigl((s,0),(x,1)\bigr)\in d(\Lambda)\Z\setminus\{0\}$, so $s\ge 2j+1$. The case $(0,s)\in\Lambda$ is analogous, using $(1,y)$.
\end{proof}

\begin{lem}\label{lem_corner}
There are $s,t\in\{1,2,\ldots,2j\}$ such that $(-s,s-1)\in\Lambda$ and $(t-1,-t)\in\Lambda$.
\end{lem}

\begin{proof}
Let $u=(0,2j)$, the upper end point of $S(j)$; see Figure~\ref{fig_corner}.

\emph{Step 1: the set $L(u)$.} Since $u\notin S(j)$ and $u-(0,\varepsilon)\in S(j)$ for small $\varepsilon>0$, Lemma~\ref{lem_edge}(a) gives $v_1=(a_0,2j)\in L(u)$. From $u-v_1=(-a_0,0)\in S(j)$ we get $-2j<a_0\le0$, and $a_0\ne0$ by Lemma~\ref{lem_axis}(b). By Lemma~\ref{lem_axis}(b) again, $v_1$ is the only element of $L(u)$ with second coordinate $2j$ (two such elements would differ by $(c,0)$ with $0<|c|<2j$). Let $v=(a,b)\in L(u)\setminus\{v_1\}$. Then $u-v=(-a,2j-b)\in S(j)$, so $a\le0$ and $b\le 2j$, and in fact $b<2j$ by the uniqueness of $v_1$. If $a=0$, then $0<b<2j$ (as $u-v\in S(j)$ forces $2j-b<2j$), contradicting Lemma~\ref{lem_axis}(b). Hence both coordinates of $u-v$ are positive and, by Lemma~\ref{lem_stair_basic}(b), $u-v\in\Int(S(j))$.

Consequently there is $\varepsilon_0$ with $0<\varepsilon_0<-a_0$ such that $u-v-(e_1,e_2)\in S(j)$ for all $v\in L(u)\setminus\{v_1\}$ and all $0\le e_1,e_2<\varepsilon_0$.

\emph{Step 2: counting near $u$.} Let $0<e,e_1,e_2<\varepsilon_0$.
\begin{itemize}
\item The point $u-(e,0)$ lies in $S(j)+v$ for all $v\in L(u)$; indeed for $v=v_1$ we have $u-(e,0)-v_1=(-a_0-e,0)\in S(j)$. These are $j$ translates, hence all of them.
\item The point $u-(0,e)$ lies in $S(j)+v$ for $v\in L(u)\setminus\{v_1\}$ and for $v=(0,0)$. These are $j$ translates, hence all of them.
\item The point $q=u-(e_1,e_2)$ lies in $S(j)+v$ for $v\in L(u)\setminus\{v_1\}$, but not in $S(j)$ (since $q_x<0$) nor in $S(j)+v_1$ (since $q_y-2j<0$). Hence $q$ lies in exactly one further translate $S(j)+w$ with $w\notin L(u)\cup\{(0,0)\}$.
\end{itemize}

\emph{Step 3: the extra translate.} Only finitely many $w\in\Lambda$ satisfy $(S(j)+w)\cap(u-[0,1]^2)\ne\emptyset$. Hence there is one such $w\notin L(u)\cup\{(0,0)\}$ with $u-(\eta,\eta)\in S(j)+w$ for a sequence $\eta=\eta_k\to0^+$. Put $p=u-w$. Then $p-(\eta_k,\eta_k)\in S(j)$ for all $k$, whereas, by Step 2, $p-(e,0)\notin S(j)$ and $p-(0,e)\notin S(j)$ for all $0<e<\varepsilon_0$. Passing to a subsequence we may assume that all $p-(\eta_k,\eta_k)$ lie in the same rectangle $[i,i+1)\times[0,2j-i)$; then $i<p_x\le i+1$ and $0<p_y\le 2j-i$. If $p_x<i+1$, then $p-(0,e)\in(i,i+1)\times(0,2j-i)\subset S(j)$ for small $e>0$, which is impossible; so $p_x=i+1$. If $p_y<2j-i$, then $p-(e,0)\in[i,i+1)\times[0,2j-i)$ for small $e>0$, again impossible; so $p_y=2j-i$. Thus $p=(s,2j+1-s)$ with $s=i+1\in\{1,\ldots,2j\}$ is an outer corner of $S(j)$, and $w=u-p=(-s,s-1)\in\Lambda$.

Applying this to the lattice $\tau(\Lambda)$ yields $t\in\{1,\ldots,2j\}$ with $(-t,t-1)\in\tau(\Lambda)$, i.e.\ $(t-1,-t)\in\Lambda$.
\end{proof}

\begin{figure}[ht]
\centering
\begin{tikzpicture}[scale=0.75]
  \draw[step=1,gray!25,very thin] (-3.5,-0.5) grid (4.5,8.5);
  \draw[->] (-3.5,0)--(4.7,0); \draw[->] (0,-0.5)--(0,8.7);
  \fill[blue!15] (0,0)--(0,4)--(1,4)--(1,3)--(2,3)--(2,2)--(3,2)--(3,1)--(4,1)--(4,0)--cycle;
  \draw[thick,blue] (0,0)--(0,4)--(1,4)--(1,3)--(2,3)--(2,2)--(3,2)--(3,1)--(4,1)--(4,0)--cycle;
  \fill[red!15,opacity=0.7] (-3,2)--(-3,6)--(-2,6)--(-2,5)--(-1,5)--(-1,4)--(0,4)--(0,3)--(1,3)--(1,2)--cycle;
  \draw[thick,red] (-3,2)--(-3,6)--(-2,6)--(-2,5)--(-1,5)--(-1,4)--(0,4)--(0,3)--(1,3)--(1,2)--cycle;
  \draw[thick,green!50!black,dashed] (-1,4)--(-1,8)--(0,8)--(0,7)--(1,7)--(1,6)--(2,6)--(2,5)--(3,5)--(3,4)--cycle;
  \fill (0,4) circle (2.5pt); \node[above right] at (0,4) {$u$};
  \node at (1.3,1) {$S(j)$};
  \node[red] at (-2.3,3) {\small $S(j)+w$};
  \node[green!50!black] at (1.7,5.5) {\small $S(j)+v_1$};
\end{tikzpicture}
\caption{The situation of Lemma~\ref{lem_corner} for $j=2$ and $\Lambda=\Lambda(1,2)$: here $v_1=(-1,4)$ and $w=(-3,2)$, i.e.\ $s=3$; the outer corner $(3,2)$ of $S(j)$ is translated to $u=(0,4)$.}\label{fig_corner}
\end{figure}

\begin{proof}[Proof of Theorem~\ref{thm_Sj}]
By Proposition~\ref{prop_sufficiency} it suffices to show that if $S(j)+\Lambda$ is an exact $j$-fold tiling, then $\Lambda=\Lambda(m,j)$ for some integer $m$. Let $s,t$ be as in Lemma~\ref{lem_corner}. Then
\[
\det\bigl((-s,s-1),(t-1,-t)\bigr)=st-(s-1)(t-1)=s+t-1
\]
is a positive multiple of $d(\Lambda)=2j+1$. As $1\le s+t-1\le 4j-1$, we get $s+t-1=2j+1=d(\Lambda)$, so $(-s,s-1)$ and $(t-1,-t)$ form a basis of $\Lambda$. In particular $\Lambda\subset\Z^2$, and
\[
-(t-1)(-s,s-1)-s(t-1,-t)=(0,s+t-1)=(0,2j+1)\in\Lambda.
\]
By Lemma~\ref{lem_axis}, $\Lambda$ contains a point $(1,m)$ with $-1<m\le 2j-1$; since $\Lambda\subset\Z^2$, $m$ is an integer. Since $\det\bigl((1,m),(0,2j+1)\bigr)=2j+1=d(\Lambda)$, the vectors $(1,m)$ and $(0,2j+1)$ form a basis of $\Lambda$, i.e.\ $\Lambda=\Lambda(m,j)$.
\end{proof}

\section{Proofs of the main results}\label{sec:proofs}

In this section $T_0$ is the triangle with vertices $(0,0)$, $(1,0)$, $(0,1)$; thus $\lambda T_0=\{p\in\Q:\sigma(p)\le\lambda\}$ and $\Int(\lambda T_0)=\{p\in(0,\infty)^2:\sigma(p)<\lambda\}$. We first associate a stair polygon with every lattice (Sections~\ref{subsec_SjLambda} and \ref{subsec_steps}), then solve the extremal problem for stair polygons (Section~\ref{subsec_extremal}), and finally combine these results with Theorem~\ref{thm_Sj} (Section~\ref{subsec_proofs}).

\subsection{The stair polygon \texorpdfstring{$S_j(\Lambda)$}{S\_j(Lambda)}}\label{subsec_SjLambda}

Let $\Lambda$ be a lattice and let $\Lambda^+=\{w\in\Lambda:(0,0)\prec w\}$. For $p\in\Q$ put
\[
N(p)=\card\{w\in\Lambda^+:\ w\le p\}.
\]
This number is finite: $w\le p$ and $\sigma(w)\ge0$ imply $w\in[-p_y,p_x]\times[-p_x,p_y]$. For $j\ge1$ define
\[
S_j(\Lambda)=\{p\in\Q:\ N(p)\le j-1\}.
\] Thus $p\in S_j(\Lambda)$ if and only if the closed quadrant $p-\Q$ contains fewer than $j$ points of $\Lambda^+$; see Figure~\ref{fig_SjLambda}. Clearly $S_1(\Lambda)\subset S_2(\Lambda)\subset\cdots$. The following reformulation of $N(p)$ is the basis of everything that follows. For $p\in\Q$, the map $w\mapsto p-w$ is a bijection from $\{w\in\Lambda^+:w\le p\}$ onto $\{q\in(p+\Lambda)\cap\Q:\ q\prec p\}$, because $p-w\in\Q\iff w\le p$, and $p-w\prec p\iff(0,0)\prec w$ by translation invariance. Hence
\begin{equation}\label{eq_N}
N(p)=\card\{q\in(p+\Lambda)\cap\Q:\ q\prec p\}\qquad(p\in\Q).
\end{equation}

\begin{figure}[ht]
\centering
\begin{tikzpicture}[scale=1.9]
  \draw[->] (-0.8,0)--(2.6,0) node[right]{$x$};
  \draw[->] (0,-0.8)--(0,2.6) node[above]{$y$};
  \fill[blue!8] (0,0)--(0,2.1)--(0.3,2.1)--(0.3,1.2)--(1.0,1.2)--(1.0,0.9)--(1.3,0.9)--(1.3,0.6)--(1.7,0.6)--(1.7,0)--cycle;
  \fill[blue!25] (0,0)--(0,1.2)--(0.3,1.2)--(0.3,0.9)--(1.0,0.9)--(1.0,0)--cycle;
  \draw[thick,blue] (0,2.1)--(0.3,2.1)--(0.3,1.2)--(1.0,1.2)--(1.0,0.9)--(1.3,0.9)--(1.3,0.6)--(1.7,0.6)--(1.7,0);
  \draw[thick,blue!70!black] (0,1.2)--(0.3,1.2)--(0.3,0.9)--(1.0,0.9)--(1.0,0);
  \draw[dashed,gray] (-0.8,0.8)--(0.8,-0.8);
  \foreach \p in {(-0.7,1.2),(-0.4,2.1),(0.3,0.9),(0.6,1.8),(1.0,-0.3),(1.3,0.6),(1.6,1.5),(1.9,2.4),(2.0,-0.6),(2.3,0.3)}{\fill \p circle (1.2pt);}
  \draw[fill=white] (0,0) circle (1.2pt);
  \fill[orange] (1.15,1.05) circle (0.9pt) node[above right]{\scriptsize $p$};
  \draw[orange,thin] (1.15,-0.8)--(1.15,1.05)--(-0.8,1.05);
  \node at (0.45,0.4) {\small $S_1$};
  \node at (0.15,1.6) {\small $S_2$};
\end{tikzpicture}
\caption{$S_1(\Lambda)\subset S_2(\Lambda)$ for the lattice $\Lambda$ spanned by $(-0.3,-0.9)$ and $(-1,0.3)$. Black dots are points of $\Lambda^+$ (the dashed line is $x+y=0$), the white dot is the origin. The quadrant $p-\Q$ of the orange point $p$ contains the two points $(0.3,0.9)$ and $(1,-0.3)$ of $\Lambda^+$, so $p\in S_3(\Lambda)\setminus S_2(\Lambda)$. Here $S_1(\Lambda)$ has one step and $S_2(\Lambda)$ has three.}\label{fig_SjLambda}
\end{figure}

\begin{lem}\label{lem_SjLambda}
Let $\Lambda$ be a lattice and $j\ge1$.
\begin{enumerate}[(a)]
\item $S_j(\Lambda)+\Lambda$ is an exact $j$-fold lattice tiling of $\R^2$; hence $|S_j(\Lambda)|=j\,d(\Lambda)$.
\item If $\lambda T_0+\Lambda$ is a $j$-fold lattice covering, then $S_j(\Lambda)\subset\lambda T_0$.
\item If $\lambda T_0+\Lambda$ is a $j$-fold lattice packing, then $\Int(\lambda T_0)\subset S_j(\Lambda)$.
\item $S_j(\Lambda)$ is a stair polygon.
\end{enumerate}
\end{lem}

\begin{proof}
(a) Let $p\in\R^2$. The set $(p+\Lambda)\cap\Q$ is infinite, it is totally ordered by $\prec$, and every element has only finitely many predecessors in it (they lie in a bounded set). Hence we may list it as $q_1\prec q_2\prec q_3\prec\cdots$, and by \eqref{eq_N}, $N(q_k)=k-1$. Thus $(p+\Lambda)\cap S_j(\Lambda)=\{q_1,\ldots,q_j\}$, and the claim follows from Lemma~\ref{lem_coset}(c) (by (d), whose proof uses only (b), $S_j(\Lambda)$ is bounded and measurable).

(b) Let $p\in\Q\setminus\lambda T_0$, so $\sigma(p)>\lambda$. By Lemma~\ref{lem_coset}(b), $(p+\Lambda)\cap\lambda T_0$ contains at least $j$ points $q$; each satisfies $q\in\Q$ and $\sigma(q)\le\lambda<\sigma(p)$, hence $q\prec p$. By \eqref{eq_N}, $N(p)\ge j$, i.e.\ $p\notin S_j(\Lambda)$.

(c) Suppose $p\in\Int(\lambda T_0)$ and $N(p)\ge j$. By \eqref{eq_N} there are $j$ points $q_1,\ldots,q_j\in(p+\Lambda)\cap\Q$ with $q_i\prec p$, so $\sigma(q_i)\le\sigma(p)<\lambda$. Choose $\eta>0$ with $\sigma(p)+2\eta<\lambda$. Then $p+(\eta,\eta)$ and $q_1+(\eta,\eta),\ldots,q_j+(\eta,\eta)$ are $j+1$ points of the same coset of $\Lambda$, all in $(0,\infty)^2$ with coordinate sum less than $\lambda$, i.e.\ all in $\Int(\lambda T_0)$. This contradicts Lemma~\ref{lem_coset}(a).

(d) If $q\le p$ are points of $\Q$, then $\{w\in\Lambda^+:w\le q\}\subset\{w\in\Lambda^+:w\le p\}$, so $N(q)\le N(p)$; thus $S_j(\Lambda)$ satisfies the property of Lemma~\ref{lem_stair_basic}(a). It contains the origin, because $w\le(0,0)$ and $\sigma(w)\ge0$ force $w=(0,0)\notin\Lambda^+$. If $\lambda$ is large enough, $\lambda T_0$ contains $j$ disjoint translates of a fundamental parallelogram of $\Lambda$, so $\lambda T_0+\Lambda$ is a $j$-fold covering and $S_j(\Lambda)\subset\lambda T_0\subset[0,\lambda]^2$ by (b). Only finitely many $w\in\Lambda^+$ satisfy $w\le p$ for some $p\in[0,\lambda]^2$. The vertical lines through these points and the horizontal lines through them, together with the coordinate axes and the lines $x=\lambda$, $y=\lambda$, cut $[0,\lambda)^2$ into finitely many half-open cells $[a,a')\times[b,b')$ on each of which $N$ is constant. Hence $S_j(\Lambda)$ is a finite union of such cells, contains $(0,0)$, and has the property of Lemma~\ref{lem_stair_basic}(a). Therefore $S_j(\Lambda)=\{(x,y)\in\Q: y<h(x)\}$, where $h(x)=\min\{y\ge0:(x,y)\notin S_j(\Lambda)\}$ is a non-increasing step function, constant on the intervals $[a,a')$ and vanishing for large $x$, with $h(0)>0$. Merging consecutive intervals on which $h$ takes the same value we obtain the representation of Definition~\ref{def_stair}.
\end{proof}

\subsection{The number of steps of \texorpdfstring{$S_j(\Lambda)$}{S\_j(Lambda)}}\label{subsec_steps}

Fix a lattice $\Lambda$ and write $S_k=S_k(\Lambda)$, $Z_k=Z(S_k)$ and $Z_k^*=Z^*(S_k)$ for $k\ge1$; put $Z_0=\emptyset$. The aim of this subsection is to prove $\card\{Z_j\}\le 2j-1$. We need to know how $N$ changes when a point is moved slightly to the left or downwards.

\begin{lem}\label{lem_left}
Let $p\in\Q$ with $p_x>0$. For all sufficiently small $\varepsilon>0$,
\[
N\bigl(p-(\varepsilon,0)\bigr)=\card\{q\in(p+\Lambda)\cap\Q:\ q\prec p,\ q_x>0\}.
\]
Similarly, if $p_y>0$, then $N\bigl(p-(0,\varepsilon)\bigr)=\card\{q\in(p+\Lambda)\cap\Q:\ q\prec p,\ q_y>0\}$ for all sufficiently small $\varepsilon>0$.
\end{lem}

\begin{proof}
By \eqref{eq_N} and translation invariance of $\prec$, $N(p-(\varepsilon,0))$ is the number of $q\in p+\Lambda$ with $q\prec p$, $q_y\ge0$ and $q_x\ge\varepsilon$. All such $q$ belong to the finite set $\{q\in(p+\Lambda)\cap\Q:q\prec p\}$; if $\varepsilon$ is smaller than every positive $x$-coordinate occurring in this set, the condition $q_x\ge\varepsilon$ is equivalent to $q_x>0$.
\end{proof}

\begin{lem}\label{lem_Z}
Let $k\ge0$ and $c\in Z_{k+1}$, and let $q_1\prec q_2\prec\cdots$ be the points of $(c+\Lambda)\cap\Q$.
\begin{enumerate}[(a)]
\item Among $q_1,\ldots,q_{k+1}$ there is a point on the $y$-axis and a point on the $x$-axis.
\item If $k=0$, then $q_1=(0,0)$, i.e.\ $c\in\Lambda$. If $k\ge1$, then $z(c):=q_{k+1}$ belongs to $S_{k+1}\cap(Z_k\cup Z_k^*)$.
\item If $c'\in Z_{k+1}\cap(c+\Lambda)$ and $c'\prec c$, then $c'\in Z_k$. In particular (as $Z_0=\emptyset$) no two points of $Z_1$ lie in the same coset of $\Lambda$.
\end{enumerate}
\end{lem}

\begin{proof}
Since $c\notin S_{k+1}$, we have $N(c)\ge k+1$, so by \eqref{eq_N} the points $q_1,\ldots,q_{k+1}$ precede $c$. Recall that $c$ has positive coordinates and that $c-(\varepsilon,0)$, $c-(0,\varepsilon)\in S_{k+1}$ for small $\varepsilon>0$ (Lemma~\ref{lem_stair_basic}(c)).

(a) By Lemma~\ref{lem_left}, $N(c-(\varepsilon,0))\le k$ means that at most $k$ predecessors of $c$ have positive $x$-coordinate. Hence one of $q_1,\ldots,q_{k+1}$ lies on the $y$-axis. The $x$-axis is treated in the same way.

(b) If $k=0$, then $q_1$ lies on both axes. Let $k\ge1$ and $z=q_{k+1}$. By \eqref{eq_N}, $N(z)=k$, so $z\in S_{k+1}\setminus S_k$, and $z\ne(0,0)$ because $(0,0)=q_1$ whenever it belongs to the coset. If $z_x>0$, then the point on the $y$-axis provided by (a) is one of $q_1,\ldots,q_k$, i.e.\ a predecessor of $z$; by Lemma~\ref{lem_left}, $N(z-(\varepsilon,0))\le k-1$, i.e.\ $z-(\varepsilon,0)\in S_k$ for small $\varepsilon>0$. Likewise, if $z_y>0$ then $z-(0,\varepsilon)\in S_k$. Since $z\notin S_k$, Lemma~\ref{lem_stair_basic}(c),(d) shows that $z\in Z_k$ if $z_x,z_y>0$, and $z\in Z_k^*$ if $z$ lies on an axis.

(c) As $c'\notin S_{k+1}$, $N(c')\ge k+1$, so $c'\notin S_k$. Every predecessor of $c'$ in $(c+\Lambda)\cap\Q$ is a predecessor of $c$, and $c'$ itself is a predecessor of $c$ with positive coordinates. By Lemma~\ref{lem_left},
\[
N\bigl(c'-(\varepsilon,0)\bigr)\le N\bigl(c-(\varepsilon,0)\bigr)-1\le k-1
\]
for small $\varepsilon>0$, and similarly $N(c'-(0,\varepsilon))\le k-1$. For $k=0$ this is impossible; for $k\ge1$ it means $c'-(\varepsilon,0),c'-(0,\varepsilon)\in S_k$, so $c'\in Z_k$ by Lemma~\ref{lem_stair_basic}(c).
\end{proof}

\begin{prop}\label{prop_steps}
For every lattice $\Lambda$ and every $j\ge1$, the stair polygon $S_j(\Lambda)$ has at most $2j-1$ steps, i.e.\ $\card\{Z_j\}\le 2j-1$.
\end{prop}

\begin{proof}
By Lemma~\ref{lem_Z}(b),(c), all points of $Z_1$ lie in the coset $\Lambda$ and no two of them lie in the same coset; hence $\card\{Z_1\}\le1$. Let $k\ge1$. By Lemma~\ref{lem_Z}(b), $c\mapsto z(c)$ maps $Z_{k+1}\setminus Z_k$ into $(Z_k\cup Z_k^*)\setminus Z_{k+1}$ (note that $Z_{k+1}\cap S_{k+1}=\emptyset$). This map is injective: if $z(c)=z(c')$ with $c\ne c'$, then $c$ and $c'$ lie in the same coset, and by Lemma~\ref{lem_Z}(c) the smaller of them belongs to $Z_k$. Since $Z_k^*$ lies on the axes, it is disjoint from $Z_k$ and from $Z_{k+1}$. Therefore
\begin{align*}
\card\{Z_{k+1}\}&=\card\{Z_{k+1}\setminus Z_k\}+\card\{Z_{k+1}\cap Z_k\}\\
&\le\card\{(Z_k\cup Z_k^*)\setminus Z_{k+1}\}+\card\{(Z_k\cup Z_k^*)\cap Z_{k+1}\}=\card\{Z_k\}+2,
\end{align*}
and the claim follows by induction.
\end{proof}

\subsection{Extremal stair polygons}\label{subsec_extremal}

\begin{lem}\label{lem_inscribed}
Let $S$ be an $r$-stair polygon with $S\subset T_0$ and $r\le 2j-1$. Then $|S|\le\frac{j}{2j+1}$, with equality if and only if $S=\frac{1}{2j+1}S(j)$.
\end{lem}

\begin{proof}
Let $S$ be as in Definition~\ref{def_stair}. Since $(x_{i+1}-\eta,y_i-\eta)\in S\subset T_0$ for small $\eta>0$, we have $x_{i+1}+y_i\le1$ for $0\le i\le r$. Put $a_i=x_{i+1}-x_i$ for $0\le i\le r$ and $a_{r+1}=1-x_{r+1}\ge0$. Then $\sum_{i=0}^{r+1}a_i=1$ and $y_i\le 1-x_{i+1}=\sum_{k=i+1}^{r+1}a_k$, so
\[
|S|=\sum_{i=0}^{r}a_iy_i\le\sum_{0\le i<k\le r+1}a_ia_k=\frac12\Bigl(1-\sum_{i=0}^{r+1}a_i^2\Bigr)\le\frac12\Bigl(1-\frac{1}{r+2}\Bigr)=\frac{r+1}{2(r+2)},
\]
by the Cauchy--Schwarz (or quadratic--arithmetic mean) inequality $1=\bigl(\sum a_i\bigr)^2\le(r+2)\sum a_i^2$. Equality holds if and only if $a_i=\frac{1}{r+2}$ for all $i$ and $y_i=1-x_{i+1}$ for all $i$, i.e.\ $x_i=\frac{i}{r+2}$ and $y_i=\frac{r+1-i}{r+2}$. As $\frac{r+1}{2(r+2)}$ increases with $r$, for $r\le 2j-1$ we obtain $|S|\le\frac{2j}{2(2j+1)}=\frac{j}{2j+1}$, with equality if and only if $r=2j-1$, $x_i=\frac{i}{2j+1}$ and $y_i=\frac{2j-i}{2j+1}$, that is, $S=\frac{1}{2j+1}S(j)$.
\end{proof}

\begin{lem}\label{lem_circumscribed}
Let $S$ be an $r$-stair polygon with $\Int(T_0)\subset S$ and $r\le 2j-1$. Then $|S|\ge\frac{2j+1}{4j}$, with equality if and only if $S=\frac{1}{2j}S(j)$.
\end{lem}

\begin{proof}
Let $S$ be as in Definition~\ref{def_stair}. For $0<x<1$ the points $(x,\eta)$, $\eta>0$ small, lie in $\Int(T_0)\subset S$; hence $x_{r+1}\ge1$. If $x_i<1$, then for small $\eta>0$ the point $(x_i+\eta,1-x_i-3\eta)$ lies in $\Int(T_0)$ and in the strip $[x_i,x_{i+1})\times\R$, so $1-x_i-3\eta<y_i$; hence $y_i\ge1-x_i$. Let $k=\max\{i:x_i<1\}\le r$, and put $b_i=x_{i+1}-x_i$ for $0\le i<k$ and $b_k=1-x_k\le x_{k+1}-x_k$. Then $\sum_{i=0}^{k}b_i=1$ and $1-x_i=\sum_{l=i}^{k}b_l$, so
\[
|S|=\sum_{i=0}^{r}(x_{i+1}-x_i)y_i\ \ge\ \sum_{i=0}^{k}b_iy_i\ \ge\ \sum_{0\le i\le l\le k}b_ib_l=\frac12\Bigl(1+\sum_{i=0}^{k}b_i^2\Bigr)\ \ge\ \frac12\Bigl(1+\frac{1}{k+1}\Bigr)=\frac{k+2}{2(k+1)}.
\]
Since $\frac{k+2}{2(k+1)}$ decreases with $k$ and $k\le r\le2j-1$, we get $|S|\ge\frac{2j+1}{4j}$. Equality forces $k=r=2j-1$, $x_{r+1}=1$ (otherwise the first inequality is strict), $b_i=\frac{1}{2j}$ and $y_i=1-x_i$ for all $i$; that is, $x_i=\frac{i}{2j}$ and $y_i=\frac{2j-i}{2j}$, i.e.\ $S=\frac{1}{2j}S(j)$. Conversely, $\frac{1}{2j}S(j)$ has area $\frac{2j+1}{4j}$.
\end{proof}

\subsection{Proofs of Theorems~\ref{main1}, \ref{main_lattices} and Corollary~\ref{cor_number}}\label{subsec_proofs}

\begin{prop}\label{prop_bounds}
Let $\Lambda$ be a lattice and $\lambda>0$.
\begin{enumerate}[(a)]
\item If $\lambda T_0+\Lambda$ is a $j$-fold lattice covering, then its density is at least $\frac{2j+1}{2}$, with equality if and only if $S_j(\Lambda)=\frac{\lambda}{2j+1}S(j)$.
\item If $\lambda T_0+\Lambda$ is a $j$-fold lattice packing, then its density is at most $\frac{2j^2}{2j+1}$, with equality if and only if $S_j(\Lambda)=\frac{\lambda}{2j}S(j)$.
\end{enumerate}
\end{prop}

\begin{proof}
(a) By Lemma~\ref{lem_SjLambda} and Proposition~\ref{prop_steps}, $S'=\lambda^{-1}S_j(\Lambda)$ is an $r$-stair polygon with $r\le2j-1$, $S'\subset T_0$ and $|S'|=\lambda^{-2}j\,d(\Lambda)$. Hence, by Lemma~\ref{lem_inscribed},
\[
\frac{|\lambda T_0|}{d(\Lambda)}=\frac{j\,|T_0|}{|S'|}\ge\frac{j}{2}\cdot\frac{2j+1}{j}=\frac{2j+1}{2},
\]
with equality if and only if $S'=\frac{1}{2j+1}S(j)$. (b) is proved in the same way, using $\Int(T_0)\subset S'$ and Lemma~\ref{lem_circumscribed}.
\end{proof}

\begin{proof}[Proof of Theorems~\ref{main1} and \ref{main_lattices}]
By affine invariance it suffices to prove Theorem~\ref{main1} for $T_0$. Let $m\in M_j$. By Theorem~\ref{thm_Sj}, $\frac{1}{2j+1}S(j)+\frac{1}{2j+1}\Lambda(m,j)$ and $\frac{1}{2j}S(j)+\frac{1}{2j}\Lambda(m,j)$ are exact $j$-fold lattice tilings. Since the outer corners of $\frac{1}{2j+1}S(j)$ lie on the line $x+y=1$, we have $\frac{1}{2j+1}S(j)\subset T_0$; since its inner corners and end points lie on $x+y=1$, we have $\Int(T_0)\subset\frac{1}{2j}S(j)$. Since every coset of a lattice meets an exact $j$-fold tile in exactly $j$ points, Lemma~\ref{lem_coset}(a),(b) show that $T_0+\frac{1}{2j+1}\Lambda(m,j)$ is a $j$-fold lattice covering of density
\[
\frac{1/2}{(2j+1)/(2j+1)^2}=\frac{2j+1}{2},
\]
and $T_0+\frac{1}{2j}\Lambda(m,j)$ is a $j$-fold lattice packing of density $\frac{1/2}{(2j+1)/(4j^2)}=\frac{2j^2}{2j+1}$. Together with Proposition~\ref{prop_bounds} (with $\lambda=1$) this proves Theorem~\ref{main1} and the inclusions ``$\supset$'' in Theorem~\ref{main_lattices}.

Conversely, let $\Lambda\in\Theta_L^j(T_0)$. By Proposition~\ref{prop_bounds}(a) with $\lambda=1$, $S_j(\Lambda)=\frac{1}{2j+1}S(j)$, and by Lemma~\ref{lem_SjLambda}(a), $\frac{1}{2j+1}S(j)+\Lambda$ is an exact $j$-fold tiling. Thus $S(j)+(2j+1)\Lambda$ is an exact $j$-fold tiling, and Theorem~\ref{thm_Sj} gives $(2j+1)\Lambda=\Lambda(m,j)$ for some integer $m$ with $\gcd(m,2j+1)=\gcd(m+1,2j+1)=1$; replacing $m$ by its residue in $\{1,\ldots,2j+1\}$ we may take $m\in M_j$. The case $\Lambda\in\Delta_L^j(T_0)$ is identical, using Proposition~\ref{prop_bounds}(b).
\end{proof}

\begin{proof}[Proof of Corollary~\ref{cor_number}]
If $T=\phi(T_0)$ with $\phi(x)=Ax+b$, then $\Lambda\mapsto A\Lambda$ is a bijection from $\Delta_L^j(T_0)$ onto $\Delta_L^j(T)$ and from $\Theta_L^j(T_0)$ onto $\Theta_L^j(T)$. By Theorem~\ref{main_lattices} and the remark preceding it (distinct residues give distinct lattices $\Lambda(m,j)$), both sets have $\card\{M_j\}$ elements. Let $n=2j+1=p_1^{\alpha_1}\cdots p_s^{\alpha_s}$. An integer $m$ satisfies $\gcd(m,n)=\gcd(m+1,n)=1$ if and only if $m\not\equiv0,-1\pmod{p_i}$ for every $i$. By the Chinese remainder theorem, the residues $m$ modulo $n$ correspond to tuples of residues modulo $p_i^{\alpha_i}$, and the number of residues modulo $p_i^{\alpha_i}$ avoiding the two (distinct, as $p_i\ge3$) classes $0$ and $-1$ modulo $p_i$ is $p_i^{\alpha_i-1}(p_i-2)$. Hence
\[
\card\{M_j\}=\prod_{i=1}^{s}p_i^{\alpha_i-1}(p_i-2)=n\prod_{i=1}^{s}\Bigl(1-\frac{2}{p_i}\Bigr).
\]
(The function $n\mapsto\card\{M_j\}$ is a generalization of Euler's $\varphi$ function.)
\end{proof}

\section*{Acknowledgment}
I am grateful to Prof.\ Chuanming Zong for suggesting this topic and for his valuable advice. I would also like to thank my friend Akanat Wetayawanich, who helped me improve this paper, and the anonymous referees for their careful reading and many helpful comments.

\end{document}